\documentclass[12pt]{amsart}
\usepackage{amsmath,amsthm,amssymb,amsfonts}
\usepackage[nobysame,abbrev,alphabetic]{amsrefs}
\usepackage{color}

\newtheorem{thm}{Theorem}[section]
\newtheorem{cor}[thm]{Corollary}
\newtheorem{lemma}[thm]{Lemma}

\theoremstyle{definition}
\newtheorem{defin}[thm]{Definition}
\newtheorem{rem}[thm]{Remark}
\newtheorem{exa}[thm]{Example}

\numberwithin{equation}{section}

\newtheorem{prop}[thm]{Proposition}

\begin{document}





\title[Isoperimetric inequality on CR-manifolds]{Isoperimetric inequality on CR-manifolds with nonnegative $Q'$-curvature}

\author{Yi Wang}
\address{Department of Mathematics, Johns Hopkins University, Baltimore MD 21218}
\email{ywang@math.jhu.edu}
\thanks{The research of the author is partially supported
by NSF grant DMS-1547878, and NSF grant DMS-1612015}
\author{Paul Yang}
\address{Department of Mathematics, Princeton University, Princeton, NJ 08540}
\email{yang@math.princeton.edu}
\thanks{The research of the author is partially supported
by NSF grant DMS-1509505.}

\date{}



\begin{abstract}
In this paper, we study contact forms on the three-dimensional Heisenberg manifold with its standard CR structure. We discover that the $Q'$-curvature, introduced by Branson, Fontana and Morpurgo \cite{BFM} on the CR three-sphere and then generalized to any pseudo-Einstein CR three manifold by Case and Yang \cite{ChY}, controls the isoperimetric inequality on such a CR-manifold. To show this, we first prove that the nonnegative Webster curvature at infinity deduces that the metric is normal, which is analogous to the behavior on a Riemannian four-manifold.

\end{abstract}

\maketitle

\section{Introduction} \label{sect:intro}
On a four dimensional manifold, the Paneitz operator $P_4$ and Branson's $Q$-curvature \cite{Branson} 
have many properties analogous to those of the Laplacian operator $\Delta_g$ and the Gaussian curvature $K_g$ on surfaces. 
The Paneitz operator is defined as $$P_g=\Delta^2+\delta(\frac{2}{3}Rg-2 Ric)d,$$
where $\delta$ is the divergence, $d$ is the differential, $R$ is the scalar curvature of $g$, and $Ric$
is the Ricci curvature tensor. 
The $Q$-curvature is defined as 
$$Q_g=\frac{1}{12}\left\{-\Delta R +\frac{1}{4}R^2 -3|E|^2 \right\},
$$
where $E$ is the traceless part of $Ric$, and $|\cdot|$ is taken with respect to the metric $g$.
The two most important properties for the pair $(P_g, Q_g)$ are that under the conformal change $g_{w}=e^{2w}g_0$, \\
1. $P_{g}$ transforms by $P_{g_w}(\cdot)=e^{-4w}P_{g_0}(\cdot)$;\\
2. $Q_{g}$ satisfies the fourth order equation
$$P_{g_{0}}w+2Q_{g_0}=2Q_{g_{w}}e^{4w}.$$
As proved by Beckner \cite{Beckner} and Chang-Yang \cite{CY}, the pair $(P_g, Q_g)$ also appears in the Moser-Trudinger inequality for higher order operators. 

On CR manifolds, it is a fundamental problem to study the existence and properties of CR invariant pairs analogous to  $(P_g, Q_g)$. Graham and Lee \cite{GL} has studied a fourth-order CR covariant operator with leading term $\Delta_b^2 +T^2$ and Hirachi \cite{Hirachi} has identified the $Q$-curvature which is related to $P$ through a change of contact form. However, although the integral of the $Q$-curvature on a compact three-dimensional CR manifold is a CR invariant, it is always equal to zero. And in many interesting cases when the CR three manifold is the boundary of a strictly pseudoconvex domains, by the work of \cite{FH}, the $Q$-curvature vanishes everywhere. As a consequence, it is desirable to search for some other invariant operators and curvature invariants on a CR manifold that are more sensitive to the CR geometry. 
The work of Branson, Fontana and Morpurgo \cite{BFM} aims to find such a pair $(P', Q')$ on the CR sphere. Later, the definition of $Q'$-curvature is generalized to all pseudo-Einstein CR manifolds by the work of Case-Yang \cite{CY} and that of Hirachi \cite{Hirachi2}.  
The construction uses the strategy of analytic continuation in dimension by Branson \cite{Branson}, restricted to the subspace of the CR pluriharmonic functions:
\begin{equation}\label{df}P'_4:=\lim_{n\rightarrow 1} \frac{2}{n-1} P_{4,n}|_{\mathcal{P}}.\end{equation}
Here $P_{4,n}$ is the fourth-order CR covariant operator that exists for every contact form $\theta$
by the work of Gover and Graham \cite{GG}. By \cite{GL}, the space of CR pluriharmonic functions $\mathcal{P}$ is always contained in the kernel of $P_{4,1}$. 
On the Heisenberg spaces with its standard contact structure, the expression of $P'$ simplifies to be 
\begin{equation}P' u=2 \Delta_b^2 u. \end{equation}

In this paper, we want explore the geometric meaning of this newly introduced conformal invariant $Q'$-curvature. 

In Riemannian geometry, a classical isoperimetric inequality on a complete simply connected surface $M^2$, called Fiala-Huber's \cite{Fiala}, \cite{Huber} isoperimetric inequality, states that
\begin{equation}\label{FialaHuber}
Vol(\Omega)\leq \frac{1}{2(2\pi-\int_{M^2}K_g^+ dv_g)} Area(\partial \Omega)^2,
\end{equation}
where $K_g^+$ is the positive part of the Gaussian curvature $K_g$. Also $\int_{M^2}K_g^+ dv_g< 2\pi$ is the sharp bound for the isoperimetric inequality to hold. 

In \cite{YW15}, the first author generalizes the Fiala-Huber's isoperimetric inequality to all even dimensions, replacing the role of the Gaussian curvature in dimension two by that of the $Q$-curvature in higher dimensions:

Let $(M^n,g)=(\mathbb{R}^n, e^{2u}|dx|^2)$ be a complete noncompact even dimensional manifold.
Let $Q^+$ and $Q^-$ denote the
positive and negative part of $Q_g$ respectively, and let $dv_g$ denote the volume form of $M$. Suppose $g= e^{2u}|dx|^2$ is a {\textit{normal}} metric, i.e.
\begin{equation}\label{normal}u(x)=
\displaystyle \frac{1}{c_n}\int_{\mathbb{R}^n} \log \frac{|y|}{|x-y|} Q_{g}(y) dv_g(y) + C,
\end{equation}
where $c_n=2^{n-2}(\frac{n-2}{2})!\pi^{\frac{n}{2}}$,
and $C$ is some constant.
If
\begin{equation}\label{assumption1}
\beta^+:= \int_{M^n}Q^{+}dv_g < c_n,
\end{equation}
and \begin{equation}\label{assumption2}
\beta^-:=\int_{M^n}Q^{-}dv_g < \infty,
\end{equation}
then $(M^n,g)$ satisfies the isoperimetric inequality with isoperimetric constant depending only on $n, \beta^+$ and
$\beta^-$.
Namely, for any bounded domain $\Omega\subset M^n$ with smooth boundary,
\begin{equation}\label{1.89}
|\Omega|_g\leq C(n, \beta^+,\beta^-) |\partial \Omega |_g^{\frac{n}{n-1}}.
\end{equation}
It is well known that if the scalar curvature is nonnegative at infinity, then one can show that the metric is a normal metric. For interested readers, the proof of such a fact when $n=4$ was given in \cite{CQY}. For higher even dimensions, one can prove by a similar manner. 

In the main result of this paper, we prove that the $Q'$-curvature and $P'$ operator are the relevant CR scalar invariant and CR covariant operator to study the isoperimetric inequalities in the CR setting. The Webster \cite{Webster} curvature at infinity imposes important geometric rigidity on the CR manifold. We also notice that the class of pluriharmonic functions $\mathcal{P}$ is the relevant subspace of functions for the conformal factor $u$.  We derive the following isoperimetric inequality on any CR three manifold with $Q'$ curvature assumptions.

\begin{thm}\label{main1}Let $(\mathbb{H}^1, e^{u}\theta)$ be a complete CR manifold, where $\theta$ denotes the standard contact form on the 
Heisenberg group $\mathbb{H}^1$ and $u$ is a pluriharmonic funcion on $\mathbb{H}^1$. Suppose additionally the $Q'$ curvature is nonnegative, the Webster scalar curvature is nonnegative at infinity and
\begin{equation}
\displaystyle \int_{\mathbb{H}^1} Q' e^{4u}  \theta\wedge d\theta< c'_1.
\end{equation}
Then the isoperimetric inequality is valid, i.e. for any bounded domain $\Omega$, 
\begin{equation}
Vol(\Omega)\leq C Area(\partial \Omega)^{4/3}.
\end{equation} 
Here $C$ depends only on the integral of the $Q'$-curvature, and $c'_1$ is the constant in the fundamental solution of $P'$ operator. (See Section \ref{sect:fundamentalsolution}.)\end{thm}

\begin{rem}
It is worth noting that the homogeneous dimension $N$ of $M^3$ is $4$. Therefore the power on the right hand side of the isoperimetric inequality is equal to $\frac{N}{N-1}=4/3$.
\end{rem}
\begin{rem}
We also remark that $c'_1$ is the critical constant for the validity of the isoperimetric inequality. In fact, there is a CR contact form $e^u \theta$ with $ \int_{\mathbb{H}^1} Q' e^{4u} \theta\wedge d\theta= c'_1$,  that does not satisfy the isoperimetric inequality. We give this example in Example \ref{exam:critical}.
\end{rem}

In fact, we have proved a stronger result.
\begin{thm}\label{main}
Suppose the $Q'$-curvature of $(\mathbb{H}^1, e^{u}\theta)$ is nonnegative. Suppose additionally the metric is normal and $u$ is a pluriharmonic function on $\mathbb{H}^1$. If 
\begin{equation}
\displaystyle \int_{\mathbb{H}^1} Q' e^{4u}  \theta\wedge d\theta< c'_1,
\end{equation} 
then $e^{4u}$ is an $A_1$ weight. 
\end{thm}

We will introduce the meaning of $A_1$ weight in Section \ref{sect:main}.

\section*{Acknowledgements}
We would like to thank the referee for valuable suggestions to improve the presentation of the paper.

\section{Fundamental solution of $P'$ operator} \label{sect:fundamentalsolution}
In this section, we compute the fundamental solution of the Paneitz operator $P'$ on the Heisenberg group $\mathbb{H}^1$. Let $p$, $q$
be two points on $\mathbb{H}^1$. $\rho$ denotes the distance function on $\mathbb{H}^1$. We show that $P' (\log \rho (q^{-1} p))$ is equal to the real part of Szeg\"o kernel. Therefore, $P'$ restricted to the space of pluriharmonic functions has the fundamental solution $\log \rho (q^{-1} p)$.

Let us first consider the case for $p=(z, t)\in \mathbb{H}^1$, and $q=(0,0)\in \mathbb{H}^1$. Note that
\begin{equation}
\begin{split}
&\Delta_b\log \rho (q^{-1} p)\\
=&\displaystyle\Delta_{b} \log (|z|^4+ t^2)^{\frac{1}{4}}\\
=&\displaystyle\frac{1}{4}(\partial_x+2y \partial_t)(\partial_x+2y\partial_t)\log (|z|^4+ t^2)\\
&+\displaystyle \frac{1}{4}(\partial_y-2x\partial_t)(\partial_y-2x\partial_t)\log (|z|^4+ t^2).
\end{split}
\end{equation}

\begin{equation}
\begin{split}
&\displaystyle(\partial_x+2y \partial_t)(\partial_x+2y\partial_t)\log (|z|^4+ t^2)\\
=&\displaystyle(\partial_x+2y \partial_t)[\frac{1}{ (|z|^4+ t^2)} (4x|z|^2+4yt)]\\
=&\displaystyle \frac{-1}{ (|z|^4+ t^2)^2} (4x|z|^2+4yt)^2+  \frac{1}{ |z|^4+ t^2} (4|z|^2 +8x^2+8y^2)\\
=&\displaystyle \frac{1}{ (|z|^4+ t^2)^2} [-16(x^2 |z|^4 + 2xyt|z|^2 + y^2 t^2 ) + 12|z|^2 (|z|^4+ t^2) ].\\
\end{split}
\end{equation}

Similarly, one can see
\begin{equation}
\begin{split}
&\displaystyle(\partial_y-2x \partial_t)(\partial_y-2x\partial_t)\log (|z|^4+ t^2)\\
=& \displaystyle\frac{1}{ (|z|^4+ t^2)^2} [-16(y^2 |z|^4 -2xyt|z|^2 + x^2 t^2 ) + 12|z|^2 (|z|^4+ t^2) ].\\
\end{split}
\end{equation}

Thus, we obtain
\begin{equation}
\begin{split}
\displaystyle\Delta_{b} \log (|z|^4+ t^2)^{\frac{1}{4}}=&
\displaystyle\frac{1}{4(|z|^4+ t^2)^2}[-16 (|z|^6 + |z|^2 t^2)+ 24|z|^2 (|z|^4+ t^2) ]\\
=& \displaystyle\frac{2|z|^2}{ |z|^4+ t^2}.\\
\end{split}
\end{equation}

We now need to compute $\displaystyle\Delta_{b}\frac{|z|^2}{|z|^4+ t^2}$.

\begin{equation}\label{fund1}
\begin{split}
&\displaystyle(\partial_x+2y \partial_t)(\partial_x+2y\partial_t)\displaystyle\frac{|z|^2}{|z|^4+ t^2}\\
=&\displaystyle(\partial_x+2y \partial_t) [\frac{2x}{ |z|^4+ t^2} +  \frac{-|z|^2}{ (|z|^4+ t^2)^2} (4x|z|^2+4yt)]\\
=&\displaystyle \frac{-2x}{ (|z|^4+ t^2)^2} (4x|z|^2+4yt)+ \frac{2}{ |z|^4+ t^2}+ \frac{2|z|^2}{ (|z|^4+ t^2)^3} (4x|z|^2 +4yt)^2\\
&\displaystyle+ \frac{-|z|^2}{ (|z|^4+ t^2)^2}(4|z|^2 +8x^2+8y^2)+ \frac{-2x}{ (|z|^4+ t^2)^2} (4x|z|^2+4yt)\\
=&\displaystyle \frac{2}{ |z|^4+ t^2} +\frac{1}{ (|z|^4+ t^2)^2}  [-8x^2 |z|^2-16xyt-12|z|^4-8 x^2|z|^2) ]\\
&\displaystyle+ \frac{32|z|^2}{ (|z|^4+ t^2)^3}  (x |z|^2  + yt)^2 .\\
\end{split}
\end{equation}

Similarly, 
\begin{equation}\label{fund2}\displaystyle
\begin{split}
&\displaystyle(\partial_y-2x \partial_t)(\partial_y-2x\partial_t)\frac{|z|^2}{|z|^4+ t^2}\\
=&\displaystyle(\partial_y-2x \partial_t) [\frac{2y}{ |z|^4+ t^2} +  \frac{-|z|^2}{ (|z|^4+ t^2)^2} (4y|z|^2-4xt)]\\
=& \displaystyle\frac{-2y}{ (|z|^4+ t^2)^2} (4y|z|^2-4xt)+ \frac{2}{ |z|^4+ t^2}+ \frac{2|z|^2}{ (|z|^4+ t^2)^3} (4y|z|^2 -4xt)^2\\
&+ \displaystyle\frac{-|z|^2}{ (|z|^4+ t^2)^2}(4|z|^2 +8x^2+8y^2)+ \frac{-2y}{ (|z|^4+ t^2)^2} (4y|z|^2-4xt)\\
=& \displaystyle\frac{2}{ |z|^4+ t^2} +\frac{1}{ (|z|^4+ t^2)^2}  [-8y^2 |z|^2+16xyt-12|z|^4-8 y^2|z|^2) ]\\
&+ \displaystyle\frac{32|z|^2}{ (|z|^4+ t^2)^3}  (y |z|^2  - xt)^2 .\\
\end{split}
\end{equation}

Therefore, by (\ref{fund1}) and (\ref{fund2}) we have
\begin{equation}
\begin{split}
\displaystyle\Delta_{b}\frac{|z|^2}{|z|^4+ t^2}
=&\displaystyle \frac{4}{ |z|^4+ t^2} +\frac{1}{ (|z|^4+ t^2)^2}  (-8 |z|^4-24|z|^4-8 |z|^4 )\\
&+\displaystyle \frac{32|z|^2}{ (|z|^4+ t^2)^3}  ( |z|^6  + |z|^2 t^2) \\
=&\displaystyle \frac{4}{ |z|^4+ t^2} -\frac{8|z|^4}{ (|z|^4+ t^2)^2} \\
=&\displaystyle 4\frac{t^2-|z|^4}{ (|z|^4+ t^2)^2}. \\
\end{split}
\end{equation}

So we've show that

\begin{equation}
\begin{split}
\displaystyle P' (\log (|z|^4+ t^2 )^{\frac{1}{4}})
=&\displaystyle2\Delta_{b}\frac{|z|^2}{|z|^4+ t^2}\\
=&\displaystyle 8\frac{t^2-|z|^4}{ (|z|^4+ t^2)^2}. \\
\end{split}
\end{equation}

Note that this is equal to the real part of Szeg\"o kernel $Re (S_{\mathbb{H}^1}(p,q))$, up to a multiplicative constant. So we've proved that $\log (|z|^4+ t^2)^{\frac{1}{4}}$ is propositional to the fundamental solution of the operator $P'$ on the space of pluriharmonic functions at point $p=(z,t)$ and $q=(0,0)$. Since the norm $\rho$ and $P'$ are both left invariant, this computation is also valid for arbitrary value of $q$. Thus we've proved that $\log(\rho(q^{-1}p))$ is propositional to the fundamental solution of $P'$. We denote $G_{\mathbb{H}^1}(u,v)= c'_1\cdot \log \rho (q^{-1}p)$.

\section{Nonnegative Webster scalar curvature at $\infty$} \label{sect:normal}

In this section, we describe the property of CR-manifolds with nonnegative Webster scalar curvature at infinity. We will see this geometric condition has a strong analytic implication. We denote the volume form $\theta\wedge d\theta$ of $\mathbb{H}^1$ by $dv$.

\begin{prop}\label{prop:normal}
Let $\theta$ be the standard contact form of the Heisenberg group $\mathbb{H}^1$, and $\hat{\theta}=e^u \theta$ be the conformal change of it. Suppose $u\in \mathcal{P}$ is a pluriharmonic function on $\mathbb{H}^1$, $\Delta_b^2 u \in L^1 (\mathbb{H}^1)$ and $\hat{\theta}$ has nonnegative Webster scalar curvature near $\infty$, i.e. $-\Delta_b u\geq |\nabla_b u|^2$. Then $\hat{\theta}$ is a {\textit{normal}}, i.e. 
\begin{equation}
u(p)= \displaystyle \int_{\mathbb{H}^1}G_{\mathbb{H}^1}(p,q)P'u(q) dv(q)+ C,
\end{equation}
where $C$ is a constant.
\end{prop}

It is proved by \cite{BFM} that the Green's function for $P'_{\mathbb{S}^{3}}$ is given by \begin{equation}G_{\mathbb{S}^3}(\zeta, \eta)=\log |1-\zeta \cdot \bar{\eta} | .\end{equation}
It satisifes the equation 
\begin{equation}
P'_{\mathbb{S}^3}G_{\mathbb{S}^3} (u, v)= S_{\mathbb{S}^3}(u, v)-\frac{1}{vol(\mathbb{S}^3)},
\end{equation}
where $S_{\mathbb{S}^3}(u, v)$ is the real part of the Szeg\"o kernel. 
We proved in section \ref{sect:fundamentalsolution} that the fundamental solution for $P'_{\mathbb{H}^1}$ is given by $\log \rho(v^{-1} u)$. We recall that the homogeneous norm on $\mathbb{H}^1$ is given by $\rho (z, t)= (|z|^4+ t^2)^{1/4}$.

\begin{defin} Given $u\in \mathcal{P}$ such that $P' u \in L^1(\mathbb{H}^1)$. Define
$$v(p):= \int_{\mathbb{H}^1} G_{\mathbb{H}^1}(p, q)P' u(q)dv(q). $$
\end{defin}
This is well-defined when $P' u \in L^1(\mathbb{H}^1)$. We want to prove that $w:=u-v$ is a linear function in $t$. 

\begin{lemma} Under the same assumption as Proposition \ref{prop:normal}, we have
$\Delta_b w =constant$.
\end{lemma}

\begin{proof}First, we observe that 
$$ P'w= P'u-P'v=0.
$$
We can then apply the mean value property to the function $\Delta_b w$ which satisfies the equation $\Delta_b (\Delta_b w)=0$. Let $K_{r}(x,y)$ denotes the Poisson kernel. We apply the Poisson integral formula to $\Delta_b w$ and derive
\begin{equation}\label{poisson}
\Delta_b  w (p)= \int_{\partial B(p, r)} \Delta_b w (q) K_r(p, q) dv(q),
\end{equation}
for arbitrary sphere $B(p, r)$ of radius $r$. Here the radius is with respect to the distance given by $\rho(\cdot)$ on $\mathbb{H}^1$.
Note that $\Delta_b u\leq -|\nabla_b u|^2\leq 0 $, and $ \Delta_b v $ tends to zero for large spheres $\partial B(p, r)$. Thus by taking $r\rightarrow \infty$,
$$\Delta_b w \leq 0,$$
at $\infty$. Thus $\Delta_b w$ is bounded from above by (\ref{poisson}) and the fact that the Poisson kernel is nonnegative. 

Now $\Delta_b w$ is bounded from above and $(\Delta_b( \Delta_b w)) =0$. Thus, analogous to the harmonic function on the Euclidean spaces, by the Liouville's theorem for $\Delta_b$ operator, we have
\begin{equation}
\Delta_b w=c_1.
\end{equation}
\end{proof}

Next we observe that besides $\Delta_b w=c_1$, $Tw$ is also a constant, because $\Delta_b^2 w+ T^2 w=0$. We denote the constant of $Tw$ by $c_2$. This allows us to show that
\begin{lemma} $w_x (x, y, t)$ is independent of $t$ variable, i.e. 
$$w_x (x, y, t)= w_x(x, y, 0).$$
\end{lemma}
\begin{proof}
We recall that 
$$X= \partial_x+ 2y \partial_t, \quad Y=\partial_y-2x\partial_t, \quad T=\partial_t.$$
Since $X$ and $T$ commute, we have
\begin{equation}
\begin{split}
0=& XTw=TX w= T(w_x+ 2y w_t)\\
=& T w_x.\\
\end{split}
\end{equation}
Thus $w_x$ is independent of $t$ variable. In other words, for any $(x, y, t)$,
$$w_x (x, y, t)= w_x(x, y, 0).$$
\end{proof}
Similarly since $Y$ and $T$ commute, $w_y$ is independent of $t$ variable.

\begin{lemma} $w_{xx}+ w_{yy}$ is independent of $t$ variable, i.e. 
$$w_{xx}(x, y, t)+ w_{yy}(x, y, t)=w_{xx}(x, y, 0)+ w_{yy}(x, y, 0).$$
\end{lemma}
\begin{proof}
This can be seen from the following computation
\begin{equation}
\begin{split}
0=&T\Delta_b w\\
=&  T[(XX+ YY)] w\\
=&T[(\partial_x+ 2y T)(\partial_x+ 2y T)+(\partial_y- 2x T)(\partial_y- 2x T)]w\\
=&T[w_{xx} +2y T\partial_x w+\partial_x(2yTw) + 2yT(2yTw) \\
&+w_{yy}-\partial_{y}(2xTw)- 2xT(\partial_y w)+ 2xT(2xTw) ].\\
\end{split}
\end{equation}
By the fact that $Tw$ is a constant, and that $T$ commutes with both $\partial_x$ and $\partial_y$, we obtain the above is equal to
$$T(w_{xx}+w_{yy}).$$ 
Thus the lemma holds.
\end{proof}

\begin{lemma}\label{lemma:subharmonic}
$$\Delta_b \partial_x w=0 \quad \Delta_b \partial_y w=0. $$
\end{lemma}
\begin{proof}
If $\Delta_b$ and $\partial_x$ commute, then since $\Delta_b w=c_1$, we prove the lemma. 
In general, $\Delta_b$ and $\partial_x$ might not commute. However, we will use the fact that $Tw$ is a constant to achieve the goal. 
\begin{equation}
\begin{split}
\Delta_b \partial_x w=&[(\partial_x+ 2y T)(\partial_x+ 2y T)\partial_x w+ (\partial_y- 2x T)(\partial_y- 2x T)\partial_x w]\\
=&w_{xxx}+\partial_x(2yT\partial_x w)+ 2yT\partial_x(\partial_x w) + 2yT (2yT \partial_x w)\\
&+ w_{xyy}- 2x T \partial_y \partial_x w -\partial_y (2x T\partial_x w) + 2xT(2x T\partial_x w)\\
=& w_{xxx}+ w_{xyy}.\\
\end{split}
\end{equation}
The last equality uses the fact that $T$ commutes with both $\partial_x$ and $\partial_y$; and the fact that $Tw$ is a constant. Thus cross terms
$$\partial_x(2yT\partial_x w); \quad 2yT\partial_x(\partial_x w); \quad 2yT (2yT \partial_x w);$$ $$2x T \partial_y \partial_x w; \quad \partial_y (2x T\partial_x w); \quad 2xT(2x T\partial_x w)$$
vanishes
\end{proof}

\begin{lemma}\label{lemma:lineargrowth}$|w_x|$ and $|w_y|$ are at most of linear growth.
\end{lemma}
\begin{proof}
\begin{equation}\label{lineargrowth}|\nabla_b w|^2= w_x^2+ w_y^2 +4 c_2^2 (x^2+y^2)- 4c_2(xw_y-yw_x).\end{equation}
The right hand side is greater than
$$ (1-\alpha) (w_x^2 + w_y^2)+ 4c_2^2(-\frac{1}{\alpha}+1)(x^2+y^2),$$ for any $\alpha>0$. Let us fix $\alpha=1/2$.
Note that $|\nabla_b w|^2 \leq2 |\nabla_b u|^2+  2|\nabla_b v|^2$ and $$ |\nabla_b u|^2\leq -\Delta_b u$$ 
near $\infty$. Also, $|\nabla_b v| $ tends to $0$ near $\infty$. Thus 
$|\nabla_b w|^2\leq -2c_1+1$ near $\infty$, where $c_1\leq 0$ is the constant value of function $\Delta_b w$.  Thus $|\nabla_b w|$ has an upper bound. It follows that $|\partial_x w|$ and $|\partial_y w|$ are at most of linear growth.
\end{proof} 

This together with Lemma \ref{lemma:subharmonic} implies that $\partial_x w$ is a linear function.  Similarly, $\partial_y w$ is also a linear function. Suppose both $\partial_x w$ and $\partial_y w$ are not constant, then $w$ is a quadratic function. Since $c_1\leq 0$, we see that $e^u \theta$ gives rise to an incomplete metric. This is a contradiction. Thus both $\partial_x
 w$ and $\partial_y w$ are constant. So $w$ is linear in both $x$ and $y$. Again, $e^u \theta$ is incomplete unless $w$ is a constant in both $x$ and $y$. In other words, $w$ only depends on $t$. On the other hand, we also have $Tw=c_2$. So $w$ is a linear function of $t$. 
We now use the assumption that the Webster scalar curvature $R$ is nonnegative to show that $w$ must be a constant.\\

To do this, we first note that by a simple computation, 
$$-\Delta_{b} (e^{c_2 t})=-4 c_2^2  (x^2+y^2)e^{c_2t} <0.$$
Also \begin{equation}\label{wconstant}
\begin{split}
 &\displaystyle Re^{2u}=-\Delta_b (e^u)\\
=& \displaystyle -\Delta_b(e^{c_2t+ v})\\
=& \displaystyle -\Delta_b (e^{c_2 t}) e^v- 2 X(e^{c_2t})X(e^v)- 2 Y(e^{c_2t}) Y(e^v)  -\Delta_b (e^{v}) e^{c_2 t}\\
=&\displaystyle  -4 c_2^2 (x^2+y^2)e^{c_2t} e^v-4c_2y e^{c_2t}X(e^v) +4c_2x e^{c_2t}Y(e^v)\\
&\hspace{60mm} -(\Delta_b v+ |\nabla_b v|^2)e^v e^{c_2 t}.\\
\end{split}
\end{equation}\\

\begin{lemma}
\begin{equation}
\frac{1}{|\partial B_r|}\int_{\partial B_r} |\nabla_b v|(x) d\sigma(x) =O(\frac{1}{r}) \quad \mbox{as} \quad r\rightarrow \infty.
\end{equation}
\end{lemma}

\begin{proof}
By a direct computation, we have
$$X(\log (|z|^4+ t^2)^{1/4})= \frac{1}{\rho^4} (|z|^2x+ ty)\leq\frac{|z|}{\rho^2}\leq \frac{1}{\rho},$$
$$Y(\log (|z|^4+ t^2)^{1/4})= \frac{1}{\rho^4} (|z|^2y- tx),$$
and
$$|\nabla_b(\log (|z|^4+ t^2)^{1/4})|= \frac{|z|}{\rho^2}\leq \frac{1}{\rho}.$$
Therefore
\begin{equation}\begin{split}
&\displaystyle\frac{1}{|\partial B_r|}\int_{\partial B_r} |\nabla_b v|(x) d\sigma(x)\\
\leq&\displaystyle \int_{\mathbb{H}^1}\frac{1}{|\partial B_r|}\int_{\partial B_r} \frac{1}{\rho (y^{-1}x)}  |Q'(y)|e^{4u(y)}dv(y) dv(x).\\
\end{split}
\end{equation}
Now we need to show 
$$\frac{1}{|\partial B_r|}\int_{\partial B_r}  \frac{1}{\rho (y^{-1}x)} d\sigma(x) \leq O(\frac{1}{r}) .$$
where $C$ is indenpent of $y$.

This is true because we can dilate and take the integration over the unit sphere.
$$\frac{1}{|\partial B_r|}\int_{\partial B_r}  \frac{1}{\rho (y^{-1}x)} d\sigma(x)
=\frac{1}{r}\cdot \frac{1}{|\partial B_1|}\int_{\partial B_1}  \frac{1}{\rho ((r^{-1}y)^{-1}x)}  d\sigma(x).$$

If $|r^{-1}y|\geq 1+\delta$ or $\leq 1-\delta$, then  it is easy to see that
$$\frac{1}{|\partial B_1|}\int_{\partial B_1}  \frac{1}{\rho ((r^{-1}y)^{-1}x)}  d\sigma (x)\leq C$$ 
for a constant $C$ independent of $x$. 

If $1-\delta\leq |r^{-1}y|\leq 1+\delta$, then  we need to use spherical coordinates to prove 
\begin{equation}\label{sphereint0}\frac{1}{|\partial B_1|}\int_{\partial B_1}  \frac{1}{\rho ((r^{-1}y)^{-1}x)}  d\sigma (x)\leq C.\end{equation} 
It is obvious that we only need to deal with the limiting case when $r^{-1}y$ is on the unit sphere $\partial B_1$. Let $r^{-1}y=(y_1, y_2, s)$ and $x=(x_1, x_2, t)$. 
Let $(r', \theta')$ be the polar coordinates centered at $(y_1, y_2)$ in the $xy$-plane (by our notation $x=(x_1, x_2, t)$, it is the $x_1x_2$-plane).
\begin{equation}\rho((y_1, y_2, s), (x_1, x_2, t))
\geq \sqrt{(x_1-y_1)^2 +(x_2-y_2)^2 } = r'.
\end{equation}
The area form of the unit sphere is given by
$$d\sigma= \sqrt{(u_{x_1}-x_2)^2+(u_{x_2}+x_1)^2}dx_1dx_2,$$
where $u(x_1,x_2)=t=\pm\sqrt{1-(x_1^2+x_2^2)^2}$.
One can directly compute that 
$$d\sigma= \sqrt{\frac{r^2(1+3r^4)}{(1-r^2)(1+r^2)}}rdrd\theta.$$
Here $(r, \theta)$ are polar coordinates of $(x_1, x_2)$ centered at $(0,0)$. It is obvious that 
$rdrd\theta=r'dr'd\theta'.$
Therefore,
\begin{equation}\label{sphereint}
\begin{split}
&\int_{\partial B_1}  \frac{1}{\rho ((r^{-1}y)^{-1}x)}  d\sigma (x)\\
\leq & 2\int_{x_1^2+x_2^2\leq 1}  \frac{1}{r'} \sqrt{\frac{r^2(1+3r^4)}{(1-r^2)(1+r^2)}} r' dr'd\theta' .\\
\end{split}
\end{equation}
Case 1: $\sqrt{y_1^2+y_2^2}<1$.\\
We can denote $\sqrt{y_1^2+y_2^2}=1-\eta$, where $\eta>0$. Then the integral \eqref{sphereint} is bounded by
\begin{equation}\label{sphereint1}
\begin{split}
&C+ 2\int_{1-\frac{\eta}{2} \leq r\leq 1}   \sqrt{\frac{r^2(1+3r^4)}{(1-r^2)(1+r^2)}}  dr'd\theta' .\\
\end{split}
\end{equation}
Here $r$ is a function of $(r', \theta')$  by the change of variable formula.
The last inequality in \eqref{sphereint} is because $r=1$ is the only singularity of such an integration.

Now, since $\sqrt{y_1^2+y_2^2}=1-\eta$ and $1-\frac{\eta}{2} \leq r\leq 1$, we have
$r'\geq \frac{\eta}{2}.$ 
Thus $dr'd\theta'= \frac{r}{r'}drd\theta\leq \frac{2}{\eta} rdrd\theta.$
Therefore
\begin{equation}
\begin{split}
&\int_{1-\frac{\eta}{2} \leq r\leq 1}   \sqrt{\frac{r^2(1+3r^4)}{(1-r^2)(1+r^2)}}  dr'd\theta' .\\
\leq& \frac{2}{\eta}\int_{1-\frac{\eta}{2} \leq r\leq 1} \sqrt{\frac{r^2(1+3r^4)}{(1-r^2)(1+r^2)}} r drd\theta.\\
\end{split}
\end{equation}
The last integral is bounded, because 
\begin{equation}2\int_{ r\leq 1} \sqrt{\frac{r^2(1+3r^4)}{(1-r^2)(1+r^2)}} r drd\theta
=2\int_{ r\leq 1} d\sigma
=|\partial B_1|<\infty.
\end{equation}
Case 2: $\sqrt{y_1^2+y_2^2}=1$.\\
Without loss of generality, we can assume that $(y_1,y_2)=(1,0)$.
We adopt the notation that $\theta'$ is the angle between the ray and the positive $x_2$-axis. Since the unit sphere on the $x_1x_2$-plane is completely on the left hand side of $(1,0)$, we have $\theta'\in [0, \pi]$.\\
Now
\begin{equation}\begin{split}
&\int _{x_1^2+x_2^2\leq 1} \frac{1}{r'} \sqrt{\frac{r^2(1+3r^4)}{(1-r^2)(1+r^2)}} r' dr'd\theta' \\
\leq& \int_{0}^{\pi} \int_{r'>\epsilon/2} \sqrt{\frac{r^2(1+3r^4)}{(1-r^2)(1+r^2)}}  dr'd\theta'+\int_{0}^{\pi} \int_{r'\leq\epsilon/2} \sqrt{\frac{r^2(1+3r^4)}{(1-r^2)(1+r^2)}} dr'd\theta'.\\
\end{split}
\end{equation}
Note that
$$\int_{0}^{\pi} \int_{r'>\epsilon/2} \sqrt{\frac{r^2(1+3r^4)}{(1-r^2)(1+r^2)}}  dr'd\theta'\leq C$$
because when $r'>\epsilon/2$, we can apply the argument in Case 1 again, using
$dr'd\theta'=\frac{r}{r'}drd\theta\leq \frac{2}{\epsilon} rdrd\theta$.\\
For $r'\leq\epsilon/2$, by a direct computation, for very small $\epsilon$, $1-r \approx r'\theta'$. 
\begin{equation}\begin{split}\label{sphereint2}
&\int_{0}^{\pi} \int_{r'\leq\epsilon/2} \sqrt{\frac{r^2(1+3r^4)}{(1-r^2)(1+r^2)}}  dr'd\theta' \\
\leq& \int_{0}^{\pi} \int_{r'\leq\epsilon/2} \sqrt{\frac{r^2(1+3r^4)}{(1-r^2)(1+r^2)}} \frac{1}{\sqrt{r'\theta'}} dr'd\theta'.\\
\end{split}
\end{equation}
Since we have $$\sqrt{\frac{r^2(1+3r^4)}{(1+r)(1+r^2)}}<C ,$$
$$  \int_{r'<\epsilon/2}\frac{1}{\sqrt{r'}} dr'<\infty, $$ and $$
 \int_{0}^{\pi}\frac{1}{\sqrt{\theta'}}  d\theta'<\infty,$$ 
 the integration in the second line of \eqref{sphereint2} is finite. This completes the proof of \eqref{sphereint0}.

\end{proof}

By a similar proof, one can show the average estimate of $|\Delta_b v|$ and $|v|$ as well.
\begin{lemma}
\begin{equation}
\frac{1}{|\partial B_r|}\int_{\partial B_r} |\Delta_b v|(x) d\sigma(x) =O(\frac{1}{r^2}) \quad \mbox{as} \quad r\rightarrow \infty.
\end{equation}
\end{lemma}

\begin{lemma}
\begin{equation}
\frac{1}{|\partial B_r|}\int_{\partial B_r} | v|(x) d\sigma(x) =O(1) \quad \mbox{as} \quad r\rightarrow \infty.
\end{equation}
\end{lemma}

So there exists a sequence of points $\{p_i\}$, $|p_i|\rightarrow \infty$, such that
\begin{equation}| v|(p_i) \leq C, \end{equation}
\begin{equation} \ |\nabla_b v|(p_i)+  |\Delta_b v|(p_i)\leq \epsilon.\end{equation}
 Moreover, we can choose $p_i$, such that they lie in the half space $c_2t \geq 0$, and away from the $t$-axis. In other words, we can require that $c_2t(p_i)\geq 0$, and that $(x(p_i), y(p_i))$ does not tend to $(0, 0)$. Here we adopt the notation that $p_i= (x(p_i), y(p_i), t(p_i))$.

When $|x|+|y|\geq L$ for some $L>0$, we have
\begin{equation}
|4c_2y e^{c_2 t} X(e^v)|\leq |y| e^{c_2t} e^v |\nabla_b v|\leq \epsilon |y| e^{c_2 t}e^v\leq 
\epsilon (x^2+y^2) e^{c_2 t}e^v;
\end{equation}

\begin{equation}
|4c_2x e^{c_2 t} Y(e^v)|\leq |x| e^{c_2t} e^v |\nabla_b v|\leq \epsilon |x| e^{c_2 t}e^v
\leq \epsilon (x^2+y^2) e^{c_2 t}e^v;
\end{equation}
 
and
$$|\Delta_b (e^{v}) e^{c_2t }|= |(\Delta_b v+|\nabla_b v|^2)e^ve^{c_2t}|\leq \epsilon e^{v}e^{c_2t}.$$

Thus  
\begin{equation}\label{a}
\begin{split}
|2 X(e^{c_2t})X(e^v)+ 2 Y(e^{c_2t}) Y(e^v)  +\Delta_b (e^{v}) e^{c_2 t}|\leq &3
\epsilon (x^2+y^2) e^{c_2 t}e^v.\\
\end{split}
\end{equation}

We want to show $c_2=0$. We prove this by contradiction. Suppose $c_2\neq 0$. Then, 
by applying \eqref{a} in (\ref{wconstant}), we obtain that
\begin{equation}
\begin{split}
 &\displaystyle Re^{2u}(p_i)\\
=&\displaystyle  -4 c_2^2 (x^2+y^2)e^{c_2t} e^v-4c_2y e^{c_2t}X(e^v) +4c_2x e^{c_2t}Y(e^v)\\
& -(\Delta_b v+ |\nabla_b v|^2)e^v e^{c_2 t}\\
\leq& -3 c_2^2 (x(p_i)^2+y(p_i)^2)e^{c_2t(p_i)} e^v(p_i),\\
\end{split}
\end{equation}
when $\epsilon$ is small enough. 

By our choice of $\{p_i\}$, $|v(p_i)|\leq C$ and $c_2t(p_i)\geq 0$ for all $i$. Thus $e^v\geq \eta>0$, and $e^{c_2t(p_i)}\geq 1$. Since $c_2\neq 0$, $-3 c_2^2 (x(p_i)^2+y(p_i)^2)e^{c_2t(p_i)} e^v(p_i)<0$, as $i\rightarrow\infty$. In fact, this quantity goes to $-\infty$ unless $(x(p_i), y(p_i))$ tends to $(0,0)$. Because if $(x(p_i)^2+y(p_i)^2)$ is bounded, then $c_2 t(p_i)\rightarrow +\infty$. This contradicts the assumption on the nonnegativity of Webster scalar curvature $R$. Therefore $c_2=0.$

This completes the proof of Proposition \ref{prop:normal}.

\section{Main Results}\label{sect:main}

To begin this section, we recall some preliminary Poincar\'{e} inequalities for Heisenberg groups $\mathbb{H}^n$ of arbitrary dimension. Let us denote the homogenous dimension by $N$. For $\mathbb{H}^n$, $N=2n+2$.
\begin{prop}\label{prop:poincare} For any ball $B$ in Heisenberg group, 
\begin{equation}\label{eqn:poincare}
\displaystyle \int_B\int_B |g(x)-g(y)|dv(x)dv(y)\leq C |B|^{\frac{N+1}{N}} \int_{2B} |\nabla_b g|dv(x) .
\end{equation}
Here $2B$ denotes the concentric ball of $B$ with double radius, and $|\cdot|$ denotes the volume with respect to the Harr measure on $\mathbb{H}^n$.
\end{prop}

In fact, the above inequality is a direct consequence of the following 1-Poincar\'e inequality.
\begin{prop} \cite{Jerison} \label{prop:1poincareJerison}
For any ball $B$ in Heisenberg group, 
\begin{equation}
\displaystyle \int_B |g(x)-g_B|dv(x) \leq C |B|^{\frac{1}{N}} \int_{2B} |\nabla_b g|dv(x).
\end{equation}
Here $2B$ denotes the concentric ball of $B$ with double radius, $g_B$ denotes the average of $g(x)$ on $B$, and $|\cdot|$ denotes the volume with respect to the Harr measure on $\mathbb{H}^n$.
\end{prop}
This implies Proposition \ref{prop:poincare} because 
\begin{equation}\begin{split}
&\int_B\int_B |g(x)-g(y)|dv(x)dv(y)\\
\leq& \int_B\int_B |g(x)-g_B|+ |g(y)-g_B|dv(x)dv(y)\\
\leq & C |B|^{\frac{N+1}{N}} \int_{2B} |\nabla_b g|dv(x).\\
\end{split}
\end{equation}
David Jerison \cite{Jerison} proved  the stronger version of 2-Poincar\'e inequality: 
\begin{equation}
\displaystyle \int_B |g(x)-g_B|^2dv(x) \leq C |B|^{\frac{2}{N}} \int_{B} |\nabla_b g|^2dv(x).
\end{equation}
The same method also implies a stronger version of 1-Poincar\'e inequality. (See \cite{HK}. )
\begin{equation}
\displaystyle \int_B |g(x)-g_B|dv(x) \leq C |B|^{\frac{1}{N}} \int_{B} |\nabla_b g|dv(x) .
\end{equation}
For the purpose of this paper, we only need the weaker statement Proposition \ref{prop:poincare}, in which the integration is over $2B$ on the right hand side of the inequality.

Given a bounded domain with smooth boundary, as a special case of the above proposition, one can take $g$ to be (a smooth approximation of) the characteristic function $\chi_{\Omega}$, and derive 
\begin{equation}\label{poincareball} 
|B\cap \Omega|\cdot |B\cap \Omega^c|\leq C |\partial \Omega \cap 2B|\cdot |B|^{\frac{N+1}{N}}.\end{equation} 
This immediately gives rise to the following corollary.

\begin{cor}\label{cor:poincare}For all balls $B\subset \mathbb{H}^n$, such that, 
$$|B\cap \Omega|\geq \frac{1}{2} |B| \quad \mbox{and} \quad |B\cap \Omega^c|\geq \frac{1}{2} |B|,$$ we have,
by (\ref{poincareball}), 
$$ |B|^{\frac{N-1}{N} }\leq C |\partial \Omega \cap 2B|. $$
\end{cor}

\begin{thm}\label{thm:A1isop}Suppose $\omega(x)\geq 0$ is an $A_1$ weight on $\mathbb{H}^n$. Namely, there exists a constant $C_0$ (independent of $B$), so that for any ball $B\subset \mathbb{H}^n$, 
\begin{equation}
\displaystyle \frac{1}{|B|}\int_B \omega(p)dv(p)\leq C_0 \inf_{z\in B} \omega(z).
\end{equation}
Then the weighted isoperimetric inequality holds for $\omega(x)$: for any domain $\Omega\subset \mathbb{H}^n$ with smooth boundary, 
\begin{equation}
\displaystyle\int_{\Omega}\omega(x)dv(x)\leq C_1 (\int_{\partial \Omega} \omega(x)^{\frac{N-1}{N}} d\sigma(x))^{\frac{N}{N-1}},
\end{equation}
where $C_1$ only depends on the $A_1$ bound $C_0$ of $\omega(x)$ and the homogeneous dimension $N=2n+2$.
\end{thm} 

We now give the proof of this theorem by Proposition \ref{prop:poincare}.

\begin{proof}
Consider a covering $\cup_{\alpha\in \Lambda} B_\alpha$ of the domain $\Omega$ such that
each $B_\alpha$ satisfies the properties:
\begin{equation}
|\frac{1}{2}B_\alpha\cap \Omega|\geq \frac{1}{2}|\frac{1}{2}B_\alpha|, \quad |\frac{1}{2}B_\alpha\cap \Omega^c|\geq \frac{1}{2}|\frac{1}{2}B_\alpha|.
\end{equation}
In other words, $|\frac{1}{2}B_\alpha\cap \Omega|$ and $|\frac{1}{2}B_\alpha\cap \Omega^c|$ are both comparable to $|\frac{1}{2}B_\alpha|$.
By Vitali covering theorem, there exists a countable subset $\cup_{i=1}^{\infty} B_i$ such that
$\Omega\subset \cup_{i=1}^{\infty} B_i$, and $\{ \frac{1}{2}B_i\}$ are mutually disjoint. 
Therefore, 
\begin{equation}
\begin{split}
\omega(\Omega)=&\displaystyle \int_\Omega\omega(x)dv(x)\leq \displaystyle \sum_{i=1}^\infty
\int_{B_i\cap\Omega}\omega(x)dv(x)\\
\leq & \displaystyle \sum_{i=1}^\infty
\int_{B_i}\omega(x)dv(x)\\
\leq & \displaystyle \sum_{i=1}^\infty C_0 |B_i | \omega(p_i)\\
\leq & C_2(n)\displaystyle \sum_{i=1}^\infty |\frac{1}{4} B_i | \omega(p_i).\\
\end{split}
\end{equation}
Here $\omega(p_i)=\inf_{x\in B_i}\omega(x)$.

By using Corolary \ref{cor:poincare} to $B=\frac{1}{4}B_i$, 
\begin{equation}
\begin{split}
\omega(\Omega)\leq &C_3 \displaystyle \sum_{i=1}^\infty |\partial \Omega\cap \frac{1}{2}B_i|^{\frac{N}{N-1}}\omega(p_i)\\
\leq & C_3 \displaystyle \sum_{i=1}^\infty ( \int_{\partial \Omega\cap  \frac{1}{2}B_i} \omega(x)^{\frac{N-1}{N}}d\sigma(x))^{\frac{N}{N-1}}\\
\leq & C_3 (\displaystyle\sum_{i=1}^\infty \int_{\partial \Omega \cap  \frac{1}{2}B_i} \omega(x)^{\frac{N-1}{N}} d\sigma(x))^{\frac{N}{N-1}}\\
\leq & C_3 (\displaystyle\int_{\partial \Omega} \omega(x)^{\frac{N-1}{N}} d\sigma(x) )^{\frac{N}{N-1}}.\\
\end{split}
\end{equation}
\end{proof}

\begin{lemma}\label{lemma4.5}
$\frac{1}{\rho(u)^\alpha}$ is an $A_1$ weight for $0<\alpha<N=2n+2$ on the Heisenberg group $\mathbb{H}^n$.
\end{lemma}
One can directly check this fact by estimating the maximal function of $\frac{1}{\rho(u)^\alpha}$.

In the following, we will give a proof of Theorem \ref{main}. Theorem \ref{main1} is then a consequence of Theorem \ref{main}, because  if $e^{4u}$ is an $A_1$ weight, by Theorem \ref{thm:A1isop}, on such a conformal Heisenberg group, the isoperimetric inequality is valid. Moreover, the isoperimetric constant depends only on the integral of the $Q'$-curvature.

\begin{proof}[Proof of Theorem \ref{main}] The PDE that the conformal factor $u$ satisfies is 
$$P' u= Q' e^{4u}.$$
Since $u$ is a pluriharmonic function, one has $\Delta_b^2u = T^2 u$. Recall that the fundamental solution of Paneitz operator $P'=2\Delta_b^2$ is given by $c'_1 \log \frac{1}{\rho (y^{-1}x)}$. By section \ref{sect:normal}, as the Webster scalar curvature at $\infty$ is nonnegative, we have the metric is normal. Namely, $u$ has an integral representation
\begin{equation}
u(x)=\frac{1}{c'_1}\int_{\mathbb{H}^1} \log \frac{\rho(y)}{\rho(y^{-1}x)} Q'(y) e^{4u(y)} dv(y)+ C.
\end{equation}
We now want to prove $e^{4u}$ is an $A_1$ weight. In other words, for any ball $B\subset \mathbb{H}^1$, 
\begin{equation}
\displaystyle M(e^{4u})(x) \leq C(\alpha) e^{4u(x)},
\end{equation}
for a.e. $x\in \mathbb{H}^1$, 
where $$ M(f)(x):= \sup_{r>0} \frac{1}{|B(x,r)|}\int_{B(x,r)}|f(y)| dv(y).
$$

Define $\alpha:= \int_{\mathbb{H}^1} Q' e^{4u} dv(x)$. By assumption, $\alpha<c'_1$. 
Note that we can assume $\alpha\neq 0$. As if $\alpha=0$, then $u$ is a constant. So the conclusion follows directly.
\begin{equation}
\begin{split}
&\displaystyle \frac{M(e^{4u})(x) }{e^{4u(x)}}\\
=& \displaystyle \sup_{r>0} 
\frac{\frac{1}{|B(x,r)|}\displaystyle \int_{B(x,r)}  \exp \left(\frac{4}{c'_1} \int_{\mathbb{H}^1} \log \frac{\rho(p)}{\rho(p^{-1}y)} Q'(p)e^{4u(p) }dv(p)\right) dv(y)   }{ \displaystyle \exp \left(\frac{4}{c'_1}  \int_{\mathbb{H}^1} \log \frac{\rho(p)}{\rho(p^{-1}x)} Q'(p)e^{4u(p) }dv(p)\right) }\\
=&\displaystyle \sup_{r>0} 
\frac{1}{|B(x,r)|}\displaystyle \int_{B(x,r)}  \exp \left(\frac{4\alpha}{c'_1} \int_{\mathbb{H}^1} \log \frac{\rho(p^{-1
}x)}{\rho(p^{-1}y)} \cdot \frac{Q'(p)e^{4u(p) } }{\alpha}dv(p)\right) dv(y). \\
\end{split}
\end{equation}
This is bounded by
\begin{equation}\label{w}
\begin{split}
&\displaystyle \sup_{r>0} 
\frac{1}{|B(x,r)|}\displaystyle \int_{B(x,r)}  \int_{\mathbb{H}^1} \left(\frac{\rho(p^{-1
}x)}{\rho(p^{-1}y)} \right)^\frac{4\alpha}{c'_1} \frac{Q'(p)e^{4u(p) } }{\alpha}dv(p) dv(y)\\ 
=&\displaystyle \sup_{r>0} 
 \int_{\mathbb{H}^1}  \frac{1}{|B(x,r)|}\displaystyle \int_{B(x,r)}  \left(\frac{\rho(p^{-1
}x)}{\rho(p^{-1}y)} \right)^\frac{4\alpha}{c'_1} dv(y) \frac{Q'(p)e^{4u(p) } }{\alpha}  dv(p). \\
\end{split}
\end{equation}
We know that by Lemma \ref{lemma4.5} $\frac{1}{\rho(x)^{\frac{4\alpha}{c'_1}}}$ is an $A_1$ weight. And so is $\frac{1}{\rho(p^{-1}x)^{\frac{4\alpha}{c'_1}}}$ for each fixed $p$.
This means 
\begin{equation}
\displaystyle \frac{\frac{1}{|B(x,r)|}\displaystyle\int_{B(x,r)} \frac{1}{\rho(p^{-1}y)^{\frac{4\alpha}{c'_1}}} dv(y) }{\frac{1}{\rho(p^{-1}x)^{\frac{4\alpha}{c'_1}}}}\leq C(\alpha),
\end{equation}
for each fixed $p$. Observe that $C$ is independent of $p$, one can substitute this inequality to the estimate (\ref{w}) and obtain that (\ref{w}) is bounded by
$$
\displaystyle  \int_{\mathbb{H}^1} C(\alpha) \frac{Q'(p)e^{4u(p) } }{\alpha}  dv(p)
=C(\alpha) .
$$
This shows that $e^{4u}$ is an $A_1$ weight. Once we have the $A_1$ property of $e^{4u}$, we can apply Theorem \ref{thm:A1isop} to it. It completes the proof of Theorem \ref{main}.
\end{proof}

Finally, we give the example, that shows $c'_1$ is the critical constant for the validity of the isoperimetric inequality.

\begin{exa}\label{exam:critical}Let $e^u\theta$ be a contact form on $\mathbb{H}^1$. And suppose $u$ is given by the following integral formula.
\begin{equation}
u(x)=\frac{1}{c'_1} \int_{\mathbb{H}^1} \log \frac{\rho (y)}{\rho (y^{-1}x)} c'_1 \delta_0 dv(y),
\end{equation}
where $\delta_0$ denotes Dirac delta function. It is obvious that the volume form $e^{4u(x)}= \frac{1}{\rho(x)^4}$ on $\mathbb{H}^1$ is not an $A_1$ weight. Moreover, such a CR manifold does not satisfy the isoperimetric inequality. This is because $e^{u} \theta= \frac{1}{\rho}\theta$ is the standard contact form on the cylinder $\mathbb{R}\times S^2 \cong \mathbb{H}^1\setminus \{(0,0,0)\}$. In particular, one can choose a sequence of rotationally symmetric annular domains $A(r_0, r)$ on $\mathbb{H}^1$, $r\rightarrow \infty$. The area of $\partial A(r_0, r)$ with respect to $e^{u} \theta$ is bounded in $r$. But the volume of $A(r_0,r)$ with respect to $e^{u} \theta$ tends to $\infty$ as $r\rightarrow \infty$. This gives a counterexample to the isoperimetric inequality. 
In this construction, $u$ is singular at the origin. But we can use the approximation argument to deal with the issue.
By choosing  $\phi_\epsilon (y)$ to be a sequence of compactly supported smooth functions approximating $c'_1\delta_0$, and defining
\begin{equation}u_\epsilon(x)=\frac{1}{c'_1} \int_{\mathbb{H}^1} \log \frac{\rho (y)}{\rho (y^{-1}x)} \phi_\epsilon (y) dv(y),\end{equation}
we construct a sequence of $u_\epsilon$ that approximates $u(x)=\log \frac{1}{\rho(x)}$ locally uniformly away from the origin.
Since $\phi_\epsilon (y)$ are compactly supported, when the annular domains $A(r_0, r)$, $r\rightarrow \infty$ are chosen such that $r_0$ is big enough (but fixed), the CR manifold $(\mathbb{H}^1, e^{u_\epsilon}\theta)$ does not satisfy the isoperimetric inequality. 
\end{exa}



\end{document}